\documentclass{amsart}

\usepackage{url}

\newtheorem{theorem}{Theorem}[section]
\newtheorem{proposition}[theorem]{Proposition}

\newtheorem{fact}[theorem]{Fact}

\theoremstyle{definition}

\theoremstyle{remark}

\newtheorem{question}[theorem]{Question}

\numberwithin{equation}{section}

\def\Ind{\setbox0=\hbox{$x$}\kern\wd0\hbox to 0pt{\hss$\mid$\hss} \lower.9\ht0\hbox to 0pt{\hss$\smile$\hss}\kern\wd0} 

\def\Notind{\setbox0=\hbox{$x$}\kern\wd0\hbox to 0pt{\mathchardef \nn=12854\hss$\nn$\kern1.4\wd0\hss}\hbox to 0pt{\hss$\mid$\hss}\lower.9\ht0 \hbox to 0pt{\hss$\smile$\hss}\kern\wd0}

\def \d {\delta}

\def \D {\Delta}

\def \DCF {\operatorname{DCF}}

\title[Algebro-geometric axioms for $\DCF_{0,m}$]{Algebro-geometric axioms for $\DCF_{0,m}$}
\author{Omar Le\'on S\'anchez}
\address{Omar Le\'on S\'anchez\\
University of Manchester\\
School of Mathematics\\
Oxford Road \\
Manchester, M13 9PL.}
\email{omar.sanchez@manchester.ac.uk}
\date{\today}

\date{\today}

\begin{document}

\begin{abstract}
We give an algebro-geometric first-order axiomatization of $\DCF_{0,m}$, the theory of differentially closed fields of characteristic zero with $m$ commuting derivations, in the spirit of the classical geometric axioms of $\DCF_0$ [Pierce and Pillay, ``A Note on the Axioms for Differentially Closed Fields of Characteristic Zero'', {\em J. Algebra}, 204, 1998].
\end{abstract}

\maketitle

\section{A brief history}\label{history}

An elegant first-order axiomatization of $\DCF_0$ (the theory of ordinary differentially closed fields of characteristic zero) was given by Blum in \cite{Bl}. Her axioms express the following: an ordinary differential field $(K,\d)$ of characteristic zero is differentially closed if and only if
\begin{enumerate}
\item [(\#)] For every pair $f$ and $g\neq 0$ of differential polynomials in one variable over $K$ such that $\operatorname{ord}(f)>\operatorname{ord}(g)$, there is $a\in K$ with $f(a)=0$ and $g(a)\neq 0$.
\end{enumerate}

In \cite{PP}, Pierce and Pillay gave an algebro-geometric axiomatization of $\DCF_0$ in terms of prolongations of affine algebraic varieties. Recall that, given an algebraically closed differential field $(K,\d)$ and a Zariski-constructible set $X$ of $K^n$, the (first) prolongation of $X$, $\tau_\d X\subseteq K^{2n}$, is the Zariski-constructible set given by the conditions
$$x\in X \quad \text{ and }\quad  \sum_{i=1}^n\frac{\partial f_j}{\partial x_i}(x)\cdot y_i +f_j^\d(x)=0\quad \text{ for } j=1,\dots,s$$
where $f_1,\dots,f_s$ are generators of the ideal of polynomials over $K$ vanishing at $X$, and each $f_j^\d$ is obtained by applying $\d$ to the coefficients of $f_j$. Note that $(a,\d a)\in \tau_\d X$ for all $a\in X$. Also, for a given $n$, we let $\pi:K^{2n}\to K^n$ be the projection onto the first $n$ coordinates. The Pierce-Pillay axioms can be stated as follows: the differential field $(K,\d)$ is differentially closed if and only if $K$ is algebraically closed and
\begin{enumerate}
\item [($\dagger$)] for every irreducible Zariski-closed set $W$ of $K^{2n}$ such that $W\subseteq \tau_\d(\pi(W))$, there is $a\in K^n$ with $(a,\d a)\in W$.
\end{enumerate}
We note that this formulation is slightly different from what appears in \cite{PP}. We will give a proof of this axiomatization via some basic results of the theory of differential kernels for ordinary differential fields in Section \ref{ordinary}.

For differentially closed fields of characteristic zero with $m$ commuting derivations, a first-order axiomatization generalizing Blum's was given by McGrail in \cite{Mc}. Another axiomatization can be found in the work of Tressl \cite{Tr}. Both of these axiomatizations of $\DCF_{0,m}$ are in terms of coherent autoreduced sets of differential polynomials, and rely heavily on deep differential-algebraic results such as Rosenfeld's lemma \cite{R} (which ultimately gives an ``algebraic'' characterization of the characteristic sets of prime differential ideals). The main motivation of this note is to give \emph{once and for all} a clear purely algebro-geometric axiomatization of $\DCF_{0,m}$. We do this in Theoreom~\ref{main} below.

Due to the commutativity of the derivations, a naive generalization of ($\dagger$) does not hold in the partial setting. For instance, consider the case of an algebraically closed field with two commuting derivations $(K,\D=\{\delta_1,\delta_2\})$. In this case, given a Zariski-constructible $X\subseteq K^n$, the (first) prolongation of $X$, $\tau_\D X\subseteq K^{3n}$, is given by the fibre-product  $\displaystyle \tau_{\d_1}X\times_X\tau_{\d_2}X$; and, given $n$, $\pi:K^{3n}\to K^n$ denotes the projection onto the first $n$ coordinates. If $W$ is the irreducible Zariski-closed set of $K^3$ given by $y=1$ and $z=x$, where $(x,y,z)$ and coordinates of $K^3$, then $\pi(W)=K$ and so $W\subseteq \tau_{\D}(\pi(W))=K^{3n}$. However, there is no $a\in K$ such that $(a,\d_1 a,\d_2 a)\in W$. Indeed, if there were such an $a$ we would get 
$$1=\d_1 \d_2 a=\d_2\d_1 a=0. $$

In a series of papers \cite{P1, P2} Pierce attempted to use differential forms to find an algebro-geometric axiomatization for $\DCF_{0,m}$. However, as he points out in \cite[\S3]{Pie}, ``differential forms are apparently not useful for this after all''. Nonetheless, in \cite{Pie}, he does manage to give an axiomatization (in arbitrary characteristic) that has a geometric flavour, though not exactly in the Pierce-Pillay spirit (it is given in terms of field extensions satisfying the \emph{differential condition} and \emph{minimal separable leaders}). On the other hand, in \cite{LS,LS2}, I established an axiomatization of $\DCF_{0,m}$ which is geometric relative to the theory $DCF_{0,m-1}$; however, ultimately these axioms are given in terms of characteristic sets of prime differential ideals, i.e., again using Rosenfeld's lemma. Personally, I find all the current axiomatizations of $\DCF_{0,m}$ unaccessible to readers with little differential-algebraic background (that is, with no knowledge of what a coherent autoreduced set is). So, I decided to write this short note to have an axiomatization of $\DCF_{0,m}$ expressed only of terms of classical algebro-geometric objects and which I hope is accessible to all readers with basic knowledge of model theory and differential fields. 

Let me finish this brief history by reminding the reader of some of the different contexts in which one can prove the existence of a model companion via a ``geometric axiomatization'':
\begin{enumerate}
\item difference fields: ACFA \cite{CH}
\item differential-difference fields: DCF$_0$A \cite{Bu} or, more generally, DCF$_{0,m}$A \cite{LS2}
\item fields in characteristic $p>0$ equipped with a derivation of the $n$-th power of the Frobenius~\cite{Kow}
\item fields with commuting Hasse-Schmidt derivations in positive characteristic~\cite{Kow2}
\item fields with free operators \cite{MS}
\item theories having a ``geometric notion of genericity'' \cite{Hi}
\end{enumerate}

{\bf Acknowledgements.} I would like to thank the anonymous referee for his/her helpful comments and suggestions on a previous version of this note.
 
\section{A quick proof in the ordinary case}\label{ordinary}

In this section we use the theory of kernels for ordinary differential fields \cite{Lando} to prove the algebro-geometric axiomatization for $\DCF_0$ of Pierce-Pillay as formulated in Section \ref{history}. 

Given an ordinary differential field $(K,\d)$ of characteristic zero, a \emph{(differential) kernel} of length $r$ over $K$ is a field extension of the form
$$L=K(a_i^j:1\leq i\leq n,0\leq j\leq r)$$
for some positive integer $n$, such that there exists a derivation 
$$D:K(a_i^j:1\leq i\leq n,0\leq j\leq r-1)\to L$$
extending $\d$ and $\d a_i^j=a_i^{j+1}$ for $j=0,\dots, r-1$. Given a kernel $(L,D)$ of length $r$, we say that $(L,D)$ has a \emph{prolongation} of length $s\geq r$ if there is a kernel $(L',D')$ of length $s$ over $K$ such that $L'$ is a field extension of $L$ and $D'$ extends $D$. Finally, we say that $(L,D)$ has a \emph{regular realization} if there is a differential field extension $(M,\d')$ of $(K,\d)$ such that $M$ is a field extension of $L$ and $\d'a_i^j=a_i^{j+1}$ for all $j=0,\dots,r-1$. 

In \cite[Proposition 3]{Lando}, Lando showed that 

\begin{fact}\label{land}
In the ordinary case, every differential kernel has a regular realization. 
\end{fact}

This yields the desired algebro-geometric axiomatization of $\DCF_{0}$. Recall that $\tau_\d$ denotes the (first) prolongation functor and, given $n$, $\pi:K^{2n}\to K^n$ is the projection onto the first $n$ coordinates.

\begin{proposition}
An ordinary differential of characteristic zero $(K,\d)$ is a model of $\DCF_{0}$ if and only if $K$ is algebraically closed and 
\begin{enumerate}
\item [($\dagger$)] for every irreducible Zariski-closed set $W$ of $K^{2n}$ such that $W\subseteq \tau_\d(\pi(W))$, there is $a\in K^n$ with $(a,\d a)\in W$.
\end{enumerate}

\end{proposition}
\begin{proof}
Assume $(K,\d)$ is differentially closed. Let $W$ be as in condition ($\dagger$), we must show that there is $a\in K^n$ such that $(a,\d a)\in W$. Let $F$ be an algebraically closed field extension of $K$ containing a generic point $(b,c)\in F^{2n}$ of $W$ over $K$.  Then, $b$ is a generic point of $\pi(W)$ over $K$. Since $(b,c)\in W\subseteq \tau_\d(\pi(W))$, by the standard argument for extending a single derivation (see \cite[Chapter~7, Theorem~5.1]{La}, for instance), we have that there is a derivation $D:K(b)\to K(b,c)$ extending $\d$ such that $\d b=c$. Thus, $L=K(b,c)$ is a differential kernel (of length 1) over $K$. By Fact \ref{land}, $L$ has a regular realization; i.e., there is a differential field extension $(M,\d')$ of $(K,\d)$ such that $\d' b=c$. Note that then $(b,\d'b)\in W$. Since $(K,\d)$ is differentially closed, we can now find the desired point $a$ in $K^n$.

The other direction is a standard argument. However, we do give the details for the partial case in Theorem \ref{main} (from which, of course, this case can be deduced).
\end{proof}

\section{Algebro-geometric axioms for $\DCF_{0,m}$}

In this section we prove our algebro-geometric axiomatization of $\DCF_{0,m}$ using the recently developed theory of differential kernels for fields with several commuting derivations \cite{GO} (which generalizes the results from the ordinary case). 

Recall that $m$ is fixed (the number of derivations). For each $\xi=(\xi_1,\dots,\xi_m)\in\mathbb N^m$, we let $\deg\xi=\xi_1+\cdots+\xi_m$. For each nonnegative integer $r$, we let 
$$\Gamma(r)=\{\xi\in \mathbb N^m: \deg \xi\leq r\}.$$
Given a differential field $(K,\D=\{\d_1,\dots,\d_m\})$ of characteristic zero, a (differential) kernel of length $r$ over $K$ is a field extension of the form
$$L=K(a_i^{\xi}:1\leq i\leq n,\, \xi\in\Gamma(r))$$
for some positive integer $n$, such that there exist derivations 
$$D_k:K(a_i^\xi:1\leq i\leq n, \,\xi\in \Gamma(r-1))\to L$$
for $k=1,\dots,m$ extending $\d_k$ and $D_ka_i^\xi=a_i^{\xi+{\bf k}}$ for all $\xi\in \Gamma(r-1)$, where ${\bf k}$ denotes the $m$-tuple whose $k$-th entry is one and zeroes elsewhere. 

Given a kernel $(L,D_1,\dots,D_k)$ of length $r$, we say it has a prolongation of length $s\geq r$ if there is a kernel $(L',D_1',\dots,D_k')$ of length $s$ over $K$ such that $L'$ is a field extension of $L$ and each $D_k'$ extends $D_k$. Finally, we say that $(L,D_1,\dots,D_k)$ has a regular realization if there is a differential field extension $(M,\D'=\{\d_1',\dots,\d_m'\})$ of $(K,\D)$ such that $M$ is a field extension of $L$ and $\d_k'a_i^\xi=a_i^{\xi+{\bf k}}$ for all $\xi\in \Gamma(r-1)$ and $k=1,\dots ,m$. 

In contrast with the ordinary case, it is not the case that every kernel has a regular realization (this is witnessed, for instance, by the example we provided in the introduction). In \cite{GO}, an upper bound $C_{r,m}^n$ was obtained for the length of a prolongation of a kernel that guarantees the existence of a regular realization. This bound depends only on the data $(r,m,n)$ and is constructed recursively. Recall that the Ackermann function $A:\mathbb N\times\mathbb N\to \mathbb N$ is a recursively defined function given as follows:
$$
A(x,y) = \begin{cases} y + 1 & \text{ if } x = 0 \\ A(x-1,1) & \text{ if } x > 0 \text{ and } y = 0 \\ A(x-1,A(x,y-1)) & \text{ if } x,y > 0. \end{cases}
$$
The natural number $C_{r,m}^n$ is recursively defined as follows: 
$$C_{0,m}^1=0, \quad\; C_{r,m}^1=A(m-1,C_{r-1,m}^1), \quad \text{ and } \quad C_{r,m}^n=C_{C_{r,m}^{n-1},m}^1.$$
For example,
$$C_{r,1}^n=r, \quad\; C_{r,2}^n=2^n r \quad \text{ and }\quad C_{r,3}^1=3(2^r-1).$$
In \cite[Theorem 3.4]{GO}, it is proved that

\begin{fact}\label{useful}
If a differential kernel $L$ of length $r$ has a prolongation of length $C_{r,m}^n$, then $L$ has a regular realization.
\end{fact}

This will yield the algebro-geometric axiomatization of $\DCF_{0,m}$. But first we need some additional notation. For a given positive integer $n$, we let 
$$\alpha(n,m)=n\cdot\binom{C_{1,m}^n+m}{m},$$
and
$$\beta(n,m)=n\cdot\binom{C_{1,m}^n-1+m}{m}.$$
We also let $\pi:K^{\alpha(n,m)}\to K^{\beta(n,m)}$ be the projection onto the first $\beta(n,m)$ coordinates; i.e., setting $(x^\xi)_{\xi\in\Gamma(C_{1,m}^n)}:=(x_i^{\xi}:1\leq i\leq n, \xi\in\Gamma(C_{1,m}^n))$ to be coordinates for $K^{\alpha(n,m)}$ then $\pi$ is the map 
$$(x^{\xi})_{\xi\in\Gamma(C_{1,m}^n)}\mapsto (x^{\xi})_{\xi\in\Gamma(C_{1,m}^n-1)}.$$
It is worth noting here that $\alpha(n,m)=n\,|\Gamma(C_{1,m}^n)|$ and $\beta(n,m)=n\,|\Gamma(C_{1,m}^n-1)|$. We will also use the projection $\psi:K^{\alpha(n,m)}\to K^{n(m+1)}$ onto the first $n(m+1)$ coordinates, that is, 
$$(x^{\xi})_{\xi\in\Gamma(C_{1,m}^n)}\mapsto (x^{\xi})_{\xi\in\Gamma(1)}.$$
Finally, we will use the embedding $\phi:K^{\alpha(n,m)}\to K^{\beta(n,m)\cdot(m+1)}$ given by 
$$(x^{\xi})_{\xi\in\Gamma(C_{1,m}^n)}\mapsto \left((x^{\xi})_{\xi\in\Gamma(C_{1,m}^n-1)},(x^{\xi+{\bf 1}})_{\xi\in\Gamma(C_{1,m}^n-1)},\dots,(x^{\xi+{\bf m}})_{\xi\in\Gamma(C_{1,m}^n-1)}\right).$$
 
Given an algebraically closed differential field $(K,\D)$, the (first) prolongation of a Zariski-constructible set $X$ of $ K^n$ is defined as the fibre-product 
$$\tau_\D X:=\tau_{\d_1}X\times_X\cdots\times_X\tau_{\d_m}X\subseteq K^{n(m+1)}.$$
In other words, $\tau_\D X$ is the Zariski-constructible set defined by the conditions
$$x\in X, \quad \text{ and }\quad  \sum_{i=1}^n\frac{\partial f_j}{\partial x_i}(x)\cdot y_{i,k} +f_j^{\d_k}(x)=0\; \text{ for } 1\leq j\leq s, \; 1\leq k\leq m$$ 
where $f_1,\dots,f_s$ are generators of the ideal of polynomials over $K$ vanishing at $X$. Note that $(a,\d_1 a,\dots,\d_m a)\in \tau_\D X$ for all $a\in X$.

\smallskip

We can now prove the desired axiomatization.

\begin{theorem}\label{main}
A differential field $(K,\D=\{\d_1,\dots,\d_m\})$ of characteristic zero is a model of $\DCF_{0,m}$ if and only $K$ is algebraically closed and
\begin{enumerate}
\item [($\sharp$)] for every irreducible Zariski-closed set $W$ of $K^{\alpha(n,m)}$ such that $\phi(W)\subseteq \tau_\D(\pi(W))$, there is $a\in K^n$ with $(a,\d_1 a,\dots,\d_m a)\in \psi(W)$.
\end{enumerate}
\end{theorem}

\begin{proof}
Assume $(K,\D)$ is differentially closed. Let $W$ be as in condition ($\sharp$), we must find a point $a\in K^n$ such that $(a,\d_1 a,\dots,\d_m a)\in \psi(W)$. Let $F$ be an algebraically closed field extension of $K$ containing a generic point $b=(b^{\xi})_{\xi\in\Gamma(C_{1,m}^n)}$ of $W$ over $K$. Then $(b^{\xi})_{\xi\in\Gamma(C_{1,m}^n-1)}$ is a generic point of $\pi(W)$ over $K$, and 
$$\phi(b)=\left((b^{\xi})_{\xi\in\Gamma(C_{1,m}^n-1)},(b^{\xi+{\bf 1}})_{\xi\in\Gamma(C_{1,m}^n-1)},\dots,(b^{\xi+{\bf m}})_{\xi\in\Gamma(C_{1,m}^n-1)}\right)\in \tau_\D (\pi(W))$$
By the standard argument for extending a single derivation (see \cite[Chapter~7, Theorem~5.1]{La}, for instance), we have that there are derivations 
$$D'_k:K(b_i^\xi:1\leq i\leq n,\xi\in \Gamma(C_{1,m}^n-1))\to K(b_i^\xi:1\leq i\leq n,\xi\in\Gamma(C_{1,m}^n))$$
for $k=1,\dots,m$ extending $\d_k$ and such that $D'_kb_i^\xi=b_i^{\xi+{\bf k}}$ for all $\xi\in\Gamma(C_{1,m}^n-1)$. Thus, $L':=K(b_i^\xi:1\leq i\leq n,\xi\in\Gamma(C_{1,m}^n))$ is a differential kernel over $K$ and, moreover, it is a prolongation of length $C_{1,m}^n$ of the differential kernel $L=K(b_i^\xi:1\leq i\leq n,\xi\in\Gamma(1))$ of length 1 with $D_k=D_k'|_{L}$. By Fact \ref{useful}, $L$ has a regular realization; i.e., there is a differential field extension $(M,\D')$ of $(K,\D)$ such that $\d_k' b^{\bf 0}=b^{\bf k}$. Then $(b^{\bf 0},\d_1' b^{\bf 0}, \dots,\d_m' b^{\bf 0})\in \psi(W)$. Since $(K,\D)$ is differentially closed, we can find the desired point $a$ in $K^n$.

For the converse, let $\rho(x)$ be a quantifier-free formula (in the language of differential rings with $m$ derivations) in variables $x=(x_1,\dots,x_t)$ over $K$ with a realization $c$ in a differential field extension $(F,\D)$ of $(K,\D)$. By passing to the algebraic closure of $F$, we may assume that $F$ is algebraically closed. We must show that $\rho$ has a realization in $K$. We can write
$$\rho(x)=\gamma(\d^\xi x:\xi\in\Gamma(r))$$
where $\gamma((x^\xi)_{\xi\in\Gamma(r)})$ is a quantifier-free formula in the language of rings over $K$, for some $r$ and where $\d^\xi:=\d_1^{\xi_1},\dots,\d_m^{\xi_m}$. If $r=0$, $\rho(x)$ has a realization in $K$ since $K$ is algebraically closed. Now assume $r>0$. Let $n:=t\cdot\binom{r-1+m}{m}$, $d:=(\d^\xi c)_{\xi\in\Gamma(r-1)}$, and
$$W:=\operatorname{Zar-loc}_K(\d^\xi d:\xi\in\Gamma(C_{1,m}^n))\subseteq F^{\alpha(n,m)}.$$
We have that $\phi(W)\subseteq \tau_\D(\pi(W))$. By ($\sharp$), there is $a=(a^\xi)_{\xi\in\Gamma(r-1)}\in K^n$ such that $(a,\d_1 a,\dots,\d_m a)\in \psi(W)$. This implies that $a^{\xi}=\d^\xi a^{\bf 0}$ for all $\xi\in \Gamma(r-1)$. Thus, 
$$(\d^\xi a^{\bf 0})_{\xi\in\Gamma(r)}\in\operatorname{Zar-loc}_K((\d^\xi c)_{\xi\in\Gamma(r)})\subseteq F^{t\cdot\binom{r+m}{m}},$$
and so, since $(\d^\xi c)_{\xi\in\Gamma(r)}$ realizes $\gamma$, we have that $(\d^\xi a^{\bf 0})_{\xi\in\Gamma(r)}$ also realizes $\gamma$. Consequently, $K\models \rho(a^{\bf 0})$.
\end{proof} 

Let us conclude with some comments on extending the above results to the positive characteristic situation. While the results from \cite{GO}, in particular Fact~\ref{useful}, were established for differential fields of characteristic zero, it is likely that the arguments of Pierce from \cite[\S4]{Pie} yield a formulation of Fact \ref{useful} that holds in positive characteristic with a possibly different bound $C_{r,m}^n$ (but still exclusively depending on the data $(r,m,n)$). This would yield a purely algebro-geometric axiomatization of the theory $\DCF_{p,m}$ for $p>0$. Note that, in the ordinary case, a positive characteristic analogue of Fact \ref{land} does hold, see \cite[Fact 1.7]{Kow}. Now, what seems to be more interesting is to work out the case of several commuting \emph{``derivations of a power of the Frobenius''}. More precisely, recall that an additive operator $\d$ on $K$ is a derivation of the $n$-th power of the Frobenius if 
$$\d(ab)=\d(a) b^{p^n}+a^{p^n}\d(b)$$
for all $a,b\in K$. If we fix the characteristic $p>0$ and an $m$-tuple $\bar n=(n_1,\dots,n_m)$ of nonnegative integers, we can consider the theory of fields of characteristic $p$ equipped with $m$ commuting additive operators $\d_1,\dots,\d_m$ such that $\d_i$ is a derivation of the $n_i$-th power of the Frobenius. Let us call such fields $\bar n$-differential fields. The ordinary case ($m=1$) was worked out by Kowalski in \cite{Kow}, where he proved that the theory of $n$-differential fields has a model companion using an algebro-geometric axiomatization. We leave the following (open) question for future work:

\begin{question}
For arbitrary $m$, does the theory of $\bar n$-differential fields in characteristic $p>0$ have a model companion? If so, is there an algebro-geometric axiomatization?
\end{question}

\bibliographystyle{plain}

\end{document}